\pdfoutput=1
\documentclass[10pt,a4paper,reqno]{amsart}
\usepackage[linktocpage = true]{hyperref}
\hypersetup{
 colorlinks = true,
 citecolor = blue,
}
\usepackage{color}
\definecolor{red}{rgb}{1,0,0}
\definecolor{green}{rgb}{0,1,0}
\definecolor{blue}{rgb}{0,0,1}
\definecolor{refkey}{gray}{.625}
\definecolor{labelkey}{gray}{.625}

\usepackage{geometry}
\usepackage{amssymb}
\usepackage{amsmath}
\usepackage{xypic}

\newcommand{\id}{\operatorname{id}}

\theoremstyle{plain}
\newtheorem{thm}[equation]{Theorem}
\newtheorem{lem}[equation]{Lemma}

\newtheorem*{theorem*}{Theorem}
\theoremstyle{definition}
\newtheorem{defn}[equation]{Definition}
\newtheorem{example}[equation]{Example}

\newtheorem{def-prop}[equation]{Definition-Proposition}
\newtheorem{prop}[equation]{Proposition}
\newtheorem*{prop*}{Proposition}
\newtheorem{prop-def}[equation]{Proposition-Definition}

\newtheorem{Rem}[equation]{Remark}

\def\id{\operatorname{id}}

\def\Hom{\operatorname{Hom}}
\def\Obj{\operatorname{Obj}}

\def\tot{\operatorname{tot}}
\def\sh{\operatorname{sh}}

\begin{document}

\title{Transgression maps for crossed modules of groupoids}

\author{Xiongwei Cai}
\address{Chern Institute of Mathematics, Nankai University, Tianjin}
\email{\href{mailto:shernvey@nankai.edu.cn}{shernvey@nankai.edu.cn}}

\begin{abstract}
Given a crossed module of groupoids $N\rightarrow G$,
we construct (1) a natural homomorphism
from the product groupoid
$\mathbb{Z}\times(N\rtimes G)\rightrightarrows N$
to the crossed product groupoid
$N\rtimes G\rightrightarrows N$
and (2) a transgression map
from the singular cohomology $H^\ast(G_\bullet,\mathbb{Z})$
of the nerve of the groupoid $G$
to the singular cohomology
$H^{\ast-1}\big((N\rtimes G)_\bullet,\mathbb{Z}\big)$
of the nerve of the crossed product groupoid $N\rtimes G$.
The latter turns out to be identical to the transgression map
obtained by Tu--Xu in their study of equivariant
$K$-theory.
\end{abstract}

\maketitle


\section{Introduction}
Transgression maps have appeared in various forms in the literature.
Brylinski--McLaughlin studied the geometric meaning of the transgression map $T_1:\Omega_G^4(\bullet)\rightarrow\Omega^3_G(G)$ for a crossed module of groups $G\xrightarrow{\id}G$ in~\cite{BM}.
Dijkgraaf--Witten used the trangression map
$T_1:H^4_G(\bullet,\mathbb{Z})\rightarrow H^3_G(G,\mathbb{Z})$
to explore the geometric correspondence between Chern-Simons
functionals and Wess-Zumino-Witten models in~\cite{DW}.

In the present paper, we investigate the transgression map
$T_1:H^\ast(G_\bullet,\mathbb{Z})\rightarrow
H^{\ast-1}((N\rtimes G)_\bullet,\mathbb{Z})$
arising from a crossed module of groupoids $N\xrightarrow{\varphi}G$,
This map was first constructed by Tu-Xu~\cite{TuXu}
who used it to produce the multiplicators which they employed
in their study of the ring structure on equivariant twisted K-theory.
Tu-Xu constructed the transgression map $T_1$
by means of a somewhat intricate combinatorial method
obscuring the key ideas.

In this paper, we give a more concrete construction.
The composition of
\begin{itemize}
\item the natural embedding of the component $(N\rtimes G)_{p-1}$
of the nerve of the groupoid $N\rtimes G$
into $\mathbb{Z}_1\times (N\rtimes G)_{p-1}$;
\item the Eilenber--MacLane map from
$\mathbb{Z}_1\times (N\rtimes G)_{p-1}$
to $(\mathbb{Z}\times (N\rtimes G))_p$;
\item the natural groupoid homomorphism
from the product groupoid
$\mathbb{Z}\times (N\rtimes G)\rightrightarrows N$
to the crossed product groupoid
$N\rtimes G\rightrightarrows N$;
\item and the projection of $N\rtimes G$ onto $G$
\end{itemize}
is a map from $(N\rtimes G)_{p-1}$ to $G_p$.
The induced map on the cohomology level
$\Phi:H^{\ast}(G_\bullet,\mathbb{Z})\rightarrow
H^{\ast-1}((N\rtimes G)_\bullet,\mathbb{Z})$
turns out to be the transgression map $T_1$
constructed by Tu-Xu~\cite{TuXu}

\subsection*{Acknowledgements}
The author would like to thank Mathieu Sti\'enon and Ping Xu for useful comments and helpful suggestions on earlier versions of the paper. Special thanks go to Gr\'{e}gory Ginot for inspiring discussions. The author is also grateful to the Pennsylvania State University for its hospitality during his stay in 2019. This work is partially supported by NSFC grant 11425104 and Nankai Zhide Foundation.

\section{Preliminary}
We recall some basic notations and facts concerning the transgression maps for crossed modules of Lie groupoids.
\subsection{Simplicial manifolds and singular cohomology}
In this subsection, we recall the definition of simplicial manifolds and singular cohomology. We refer the reader to~\cite{Dupont,Xu} for more details.

Consider the \emph{simplicial category} $\Delta$ whose objects
are the finite ordered sets $[n]=\{0<1<\cdots<n\}$ (for $n\in\mathbb{N}$)
and whose morphisms are the nondecreasing functions.

There are two kinds of basic morphisms in $\Delta$:
the face maps and the degeneracy maps.
For each $i\in\{0,1,\cdots,n\}$, the \emph{face map}
$\varepsilon^{i}:[n-1]\rightarrow[n]$ is defined by
\[ \varepsilon^{i}(j)=\begin{cases}
j & \text{if } j<i,\\
j+1 & \text{if } j\geqslant i .\end{cases} \]
For each $i\in\{0,1,\cdots,n\}$, the \emph{degeneracy map}
$\eta^{i}:[n+1]\rightarrow[n]$ is defined by
\[ \eta^{i}(j)=\begin{cases}
j & \text{if } j\leqslant i, \\
j-1 & \text{if } j>i .\end{cases} \]
The face and degeneracy maps satisfy the commutation relations
\[ \varepsilon^{j}\varepsilon^{i}=\varepsilon^{i}\varepsilon^{j-1}
\quad\text{if }i<j,
\qquad
\eta^{j}\eta^{i}=\eta^{i}\eta^{j+1}\quad\text{if }i\leqslant j, \]
and
\[ \eta^{j}\varepsilon^{i}=\begin{cases}
\varepsilon^{i}\eta^{j-1} & \text{if } i<j, \\
\id & \text{if } i\in\{j,j+1\}, \\
\varepsilon^{i-1}\eta^{j} & \text{if } i>j+1 .
\end{cases} \]

A \emph{simplicial set} is defined to be a contravariant functor from the simplicial category $\Delta$ to the category of sets.
A \emph{simplicial (topological) space} is defined to be a contravariant functor from the simplicial category $\Delta$ to the category of topological spaces.
A \emph{simplicial manifold} is defined to be a contravariant functor from the simplicial category $\Delta$ to the category of manifolds.

The notation $M_\bullet$ is often used to denote a simplicial
set, space, or manifold and the set, space, or manifold
corresponding to the object $[k]$ of $\Delta$
is customarily denoted $M_k$.
The morphisms
$\overline{\varepsilon}^i:M_n\to M_{n-1}$ and $\overline{\eta}^i:M_n\to M_{n+1}$, images of the face and degeneracy maps $\varepsilon^i:[n-1]\to[n]$
and $\eta^i:[n+1]\to[n]$,
are called face and degeneracy operators.

Similarly, \emph{cosimplical sets, spaces, and manifolds} are defined
to be covariant functors from the simplicial category $\Delta$ to the categories of sets, topological spaces, and manifolds, respectively.
The notation $M^\bullet$ is often used to denote a cosimplicial
set, space, or manifold.
The morphisms $\underline{\varepsilon}^i:M^{n-1}\to M^n$
and $\underline{\eta}^i:M^{n+1}\to M^n$,
images of the face and degeneracy maps
$\varepsilon^i:[n-1]\to[n]$ and $\eta^i:[n+1]\to[n]$,
are called face and degeneracy operators.

\begin{example}
  Given any Lie groupoid $G\rightrightarrows G_0$, there is associated a simplicial manifold $G_\bullet$, called the \emph{nerve}, with
\[ G_k=\{(g_1,g_2,\cdots,g_k)|g_i\in G,1\leqslant i\leqslant k, g_1g_2\cdots g_k\ \mbox{makes sense}\} \]
being the manifold of composable $n$-tuples. The face operators $\overline{\varepsilon}^i:G_{n}\rightarrow G_{n-1}$ are given by
\[
\overline{\varepsilon}^i(g_1,\cdots,g_{n})=\begin{cases}
  (g_2,\cdots,g_{n}) & \text{if } i=0,\\
  (g_1,\cdots,g_{i-1},g_i g_{i+1},g_{i+2},\cdots,g_{n})
& \text{if } 0<i<n,\\
  (g_1,\cdots,g_{n-1}) & \text{if } i=n,
\end{cases}
\]
while the degeneracy operators $\overline{\eta}^i:G_{n}\rightarrow G_{n+1}$ are given by
\[
\overline{\eta}^i(g_1,\cdots,g_{n})=\begin{cases}
  (\varepsilon(s(g_1)),g_1,\cdots,g_{n}) & \text{if } i=0,\\
  (g_1,\cdots,g_{i},\varepsilon(t(g_i)),g_{i+1},\cdots,g_{n})
& \text{if } 0<i\leqslant n,
\end{cases}
\]
where $\varepsilon:G_0\rightarrow G$ is the unit map of the groupoid $G\rightrightarrows G_0$.
\end{example}

\begin{example}
\begin{itemize}
\item[(1)] For any $n\in \mathbb{Z}$, the \emph{standard n-simplex} $\Delta^n$ defined by
  \[ \Delta^n:=\{(t_0,\cdots,t_n)|\sum_{i=0}^{n}t_i=1,t_0\geqslant 0,\cdots,t_n\geqslant 0\} \]
  can be thought of as a simplicial manifold $(\Delta^n)_\bullet$:
  \[ (\Delta^n)_k:=\begin{cases}
    \Delta^{k+1} & \text{if }k<n,\\
    \{(t_0,\cdots,t_n,0,\cdots,0)|\sum_{i=0}^{n}t_i=1,t_0\geqslant 0,\cdots,t_n\geqslant 0\} & \text{if } k\geqslant n,
  \end{cases} \]
 and the face operators $\overline{\varepsilon}^i:(\Delta^n)_k\rightarrow (\Delta^n)_{k-1}$ and the degeneracy operators $\overline{\eta}^i:(\Delta^n)_k\rightarrow (\Delta^n)_{k+1}$,$i=0,1,\cdots,k$, are given repectively by
 \begin{align*}
   \overline{\varepsilon}^i(t_0,\cdots,t_{k+1})&=(t_0,\cdots,t_i+t_{i+1},\cdots,t_{k+1}),\\
   \overline{\eta}^i(t_0,\cdots,t_{k+1})&=(t_0,\cdots,t_i,0,t_{i+1},\cdots,t_{k+1}).
 \end{align*}
 \item[(2)] Furthermore, the collection of standard simplices $\Delta^\bullet$ can be viewed as a cosimplicial manifold. The face operators $\underline{\varepsilon}^i:\Delta^{k-1}\rightarrow \Delta^{k}$ and the degeneracy operators $\underline{\eta}^i:\Delta^{k+1}\rightarrow \Delta^{k}$,$i=0,1,\cdots,k$, are given repectively by
\begin{align*}
\underline{\varepsilon}^i(t_0,\cdots,t_{k-1})&=(t_0,\cdots,t_{i-1},0,t_{i},\cdots,t_{k-1}),\\
\underline{\eta}^i(t_0,\cdots,t_{k+1})&=(t_0,\cdots,t_i+t_{i+1},\cdots,t_{k+1}).
 \end{align*}
 \end{itemize}
\end{example}

Given a simplicial manifold $M_\bullet$, a \emph{(smooth) singular $q$-simplex} in $M_p$ is defined to be a smooth map $\sigma$ from the standard $q$-simplex $\Delta^q$ to $M_p$. Denote by $S_q^\infty(M_p)$ the set of all smooth singular $q$-simplices. It is clear that the cosimplicial manifold structure on $\Delta^\bullet$ induces a simplicial manifold structure on $S_\bullet^\infty(M_p)$. The face and degeneracy operators of the latter, still denoted by $\overline{\varepsilon}^i$ and $\overline{\eta}^i$ by abuse of notations, are obtained by pullback along $\underline{\varepsilon}^i$ and $\underline{\eta}^i$:
\begin{align*}
  \overline{\varepsilon}^i:S_q^\infty(M_p)\rightarrow S_{q-1}^\infty(M_p),\qquad &\overline{\eta}^i:S_q^\infty(M_p)\rightarrow S_{q+1}^\infty(M_p),\\
  \overline{\varepsilon}^i(\sigma)=(\underline{\varepsilon}^i)^\ast\sigma=\sigma\circ\underline{\varepsilon}^i,\qquad &\overline{\eta}^i(\sigma)=(\underline{\eta}^i)^\ast\sigma=\sigma\circ\underline{\eta}^i.
\end{align*}
Let $C_q(M_p,\mathbb{Z}):=\mathbb{Z}[S_q^\infty(M_p)]$ be the free $\mathbb{Z}$-module generated by singular $q$-simplices, i.e. elements of the finite formal sum
\[ \sum_{\sigma\in S_q^\infty(M_p)}a_\sigma\sigma,\quad a_\sigma\in \mathbb{Z} .\]
$C_\bullet(M_p,\mathbb{Z})$ becomes a chain complex under the boundary operator
\[ \delta:C_q(M_p,\mathbb{Z})\rightarrow C_{q-1}(M_p,\mathbb{Z}) \]
defined by the alternate sum
\[ \delta:=\sum_{i=0}^q(-1)^i\overline{\varepsilon}^i=\sum_{i=0}^q(-1)^i(\underline{\varepsilon}^i)^\ast .\]
In fact $\delta$ can be viewed as the vertical boundary operator of a double complex $C_\bullet(M_\bullet,\mathbb{Z})$. And the horizontal boundary operator of the double complex
\[ \partial:C_q(M_p,\mathbb{Z})\rightarrow C_{q}(M_{p-1},\mathbb{Z}) \]
is given by the alternate sum
\[ \partial:=\sum_{j=0}^p(-1)^j(\overline{\varepsilon}^j)_\ast ,\]
where $(\overline{\varepsilon}^j)_\ast:\ C_q(M_p,\mathbb{Z})\rightarrow C_{q}(M_{p-1},\mathbb{Z})$ is the chain map induced by the face operators $\overline{\varepsilon}^j:M_p\rightarrow M_{p-1}$. The total differential is $(-1)^p\delta+\partial$.
The corresponding hyperhomology is denoted by $H_\bullet(M_\bullet,\mathbb{Z})$, and called the \emph{(smooth) singular homology} of the simplicial manifold $M_\bullet$ with coefficients in $\mathbb{Z}$.

Dually, let $C^q(M_p,\mathbb{Z})$ be the space of smooth singular $q$-cochains:
\[ C^q(M_p,\mathbb{Z}):=\Hom_{\mathbb{Z}}(C_q(M_p,\mathbb{Z}),\mathbb{Z}) .\]
Then $C^\bullet(M_\bullet,\mathbb{Z})$ becomes a double cochain complex, with vertical coboundary operator $\delta^\ast$ and horizontal coboundary operator $\partial^\ast$. The total differential is denoted by
\[ D=(-1)^p\delta^\ast+\partial^\ast .\]
The corresponding hypercohomology is denoted by $H^\bullet(M_\bullet,\mathbb{Z})$, and called the \emph{(smooth) singular cohomology} of the simplicial manifold $M_\bullet$ with coefficients in $\mathbb{Z}$ --- see~\cite{Dupont,Xu}.

\subsection{Crossed modules of Lie groupoids and transgression maps}
In this subsection, we recall the definition of crossed modules of Lie groupoids. Then we rewrite Tu \& Xu's result on transgression maps in terms of singular cohomology. We refer the reader to \cite{MoMr,TuXu} for more details.

First, we recall the definition of an action of a groupoid.
\begin{defn}
A \emph{(right) action of a groupoid} $G\rightrightarrows G_0$ on a space $X$ is given by
\begin{itemize}
\item[(1)] a map $J:X\rightarrow G_0$, called the momentum map;
\item[(2)] a map $(x,g)\mapsto x^g$
from $X\times_{G_0}G:=\{(x,g)\in X\times G|J(x)=t(g)\}$ to $X$ such that
\[ J(x^g)=s(g), \qquad x^{gh}=(x^g)^h, \quad\text{and}\quad x^{J(x)}=x \]
whenever $J(x)=t(g)$ and $s(g)=t(h)$.
\end{itemize}
When endowed with such an action, we say that the space $X$
is a \emph{(right) $G$-space}.
\end{defn}
Then the \emph{transformation groupoid} (also called \emph{crossed product groupoid}) $X\rtimes G\rightrightarrows X$ is defined as follows: $(X\rtimes G)_0=X$, $(X\rtimes G)_1=X\times_{G_0}G$, and the source, target and multiplication maps are given by
\[ s(x,g)=x^g,\quad t(x,g)=x,\quad (x,g)(x^g,h)=(x,gh) .\]

Suppose $N\rightrightarrows N_0$ and $G\rightrightarrows G_0$ are groupoids. We say that \emph{$G\rightrightarrows G_0$ acts on $N\rightrightarrows N_0$ by automorphisms} if both $N_0$ and $N_1$ are (right) $G$-spaces and the actions are compatible in the following sense:
  \begin{itemize}
    \item[(1)] the source and target maps of $N$ are $G$-equivariant,
    \item[(2)] $x^gy^g=(xy)^g$ for all $(x,y,g)\in N\times N \times G$ whenever either side makes sense.
  \end{itemize}
\begin{defn}
A \emph{crossed module of groupoids} is a groupoid homomorphism
  \[
\xymatrix{N\ar[r]^{\varphi}\ar@<-.5ex>[d]\ar@<.5ex>[d] & G\ar@<-.5ex>[d]\ar@<.5ex>[d]\\
N_0\ar[r]^{=} & G_0
}
\]
   where $N\rightrightarrows N_0$ is a bundle of groups, together with an action of $G\rightrightarrows G_0$ on $N\rightrightarrows N_0$ by automorphisms such that
  \begin{itemize}
    \item[(1)] $\varphi(x^g)=\varphi(x)^g$  for all $x\in N$ and $g\in G$ such that $x^g$ makes sense,
    \item[(2)] $x^{\varphi(y)}=x^y$ for all composable pairs $(x,y)\in N_2$,
  \end{itemize}
  where $\varphi(x)^g:=g^{-1}\varphi(x)g$ and $x^y:=y^{-1}xy$.
\end{defn}
Given a crossed module of groupoids $N\xrightarrow{\varphi}G$, let $N\rtimes G\rightrightarrows N$ be the associated crossed product groupoid. Its nerve $(N\rtimes G)_{\bullet}$ is indeed isomorphic to $N\rtimes G_\bullet$ as spaces, where
\[ N\rtimes G_l:=\{(x,g_1,\cdots,g_l)\in N\times G^l|t(x)=t(g_1),g_1\cdots g_l\ \mbox{makes sense}\} .\]
The bijection is given as follows
\[ (N\rtimes G)_l \ni
\big((x,g_1),(xg_1,g_2),\cdots,(xg_1\cdots g_{l-1},g_l)\big)
\longleftrightarrow (x,g_1,\cdots,g_l)
\in N\rtimes G_l. \]

As a bridge between $H^\ast(G_\bullet;\mathbb{Z})$, the singular cohomology of the nerve $G_\bullet$, and $H^\ast((N\rtimes G)_\bullet,\mathbb{Z})$, the singular cohomology of the nerve $(N\rtimes G)_\bullet$, Tu-Xu~\cite{TuXu} constructed a transgression map
\[
T_{1}:H^{\ast}(G_\bullet,\mathbb{Z})\rightarrow H^{\ast-1}((N\rtimes G)_\bullet,\mathbb{Z}),
\]
which is given on the cochain level by the following theorem:
\begin{thm}[Tu $\&$ Xu \cite{TuXu}]\label{TuXuResult}
Let $N\xrightarrow{\varphi}G$ be a crossed module of groupoids,
and let $N\rtimes G\rightrightarrows N$ be
the corresponding crossed product groupoid.
The alternating sum
$T_1=\sum_{i=0}^{n-1}(-1)^i\tilde{f}_i^\ast$
of the pullbacks by the maps
$\tilde{f}_i:N\rtimes G_{p-1}\rightarrow G_p$ given by
\[ \tilde{f}_i(x,g_1,\cdots,g_{p-1})=(g_1,\cdots,g_i,\varphi(x)^{g_1\cdots g_i},g_{i+1},\cdots,g_{p-1}) \]
defines a cochain map
\[ T_1:C^q(G_p,\mathbb{Z})\rightarrow
C^q((N\rtimes G)_{p-1},\mathbb{Z}) \]
and hence a transgression map
\[ T_{1}:H^{\ast}(G_\bullet,\mathbb{Z})\rightarrow
H^{\ast-1}((N\rtimes G)_\bullet,\mathbb{Z}) .\]
\end{thm}

\section{Natural groupoid homomorphisms}

Given any groupoid $G\rightrightarrows G_{0}$, there is a crossed
module $SG\xrightarrow{i} G$ associated to $G\rightrightarrows G_{0}$: $SG:=\{g\in G|s(g)=t(g)\}$
is the space of closed loops in $G$, $i$ is the inclusion of $SG$
to $G$, and $G$ acts on $SG$ (from right) by conjugation ($a\cdot g:=g^{-1}ag,\ a\in SG, g\in G$).
The crossed product groupoid $SG\rtimes G\rightrightarrows SG$ is called the \emph{inertia groupoid} of $G$, denoted by $\Lambda G$.

Given a crossed module $N\xrightarrow{\varphi}G$, since $\varphi$
maps $N$ into $SG$, we have a natural crossed module morphism from
$N\xrightarrow{\varphi}G$ to $SG\xrightarrow{i}G$:
\[ \xymatrix{N\ar[r]^{\varphi}\ar[d]_{\varphi} & SG\ar[d]^{i}\\
G\ar[r]^{=} & G } \]

\begin{prop}
Let $G\rightrightarrows G_0$ be any groupoid. The group of integers $\mathbb{Z}$ can be viewed as a groupoid $\mathbb{Z}\rightrightarrows\bullet$ (still denoted by $\mathbb{Z}$ in the sequel by abuse of notation). There is a natural groupoid
structure on the hom-set $\Hom(\mathbb{Z},G)$ of groupoid homomorphisms.
And the inertia groupoid $\Lambda G$ is isomorphic to $\Hom(\mathbb{Z},G)$.\end{prop}
\begin{proof}
In fact, given any two groupoids $H$ and $G$, there is a natural
groupoid structure on $\Hom(H,G)$:
\begin{itemize}
\item the objects of $\Hom(H,G)$ are groupoid homomorphisms from $H$ to $G$, or equivalently, functors from $H$ to $G$ if $H$ and $G$ are viewed as categories;
\item the arrows in $\Hom(H,G)$ are natural transformations between functors;
\item the source and target of an arrow $\eta:\beta\Leftarrow \alpha$ are $\alpha$ and $\beta$, respectively;
\item the unit map sends a functor $\alpha\in\Obj(\Hom(H,G))$ to a natural isomorphism $1_{\alpha}:\alpha\Leftarrow \alpha$, whose component at $x\in H_{0}$ is the identity arrow $1_{\alpha(x)}\in G$, i.e.
$(1_{\alpha})_{x}=1_{\alpha(x)}$;
\item the inverse map sends a natural transformation $\eta:\beta\Leftarrow \alpha$
to the natural transformation $\eta^{-1}:\alpha\Leftarrow \beta$, whose component at $x\in H_{0}$ is the inverse of the component of $\eta$
at $x\in H_{0}$, i.e. $(\eta^{-1})_{x}=(\eta_{x})^{-1}$;
\item the multiplication of $\eta^{\prime}:\gamma\Leftarrow \beta$ and $\eta:\beta\Leftarrow \alpha$ is the composition $\eta^{\prime}\circ\eta:\gamma\Leftarrow \alpha$, whose component at $x\in H_{0}$ is the multiplication of the components of $\eta^{\prime}$ and $\eta$ at $x\in H_{0}$, i.e. $(\eta^{\prime}\circ\eta)_{x}=\eta_{x}^{\prime}\eta_{x}$.
\end{itemize}
We can easily verify that $\Hom(H,G)$ becomes a groupoid with the struture maps above.

Next, we focus on the groupoid $\Hom(\mathbb{Z},G)$.
\begin{itemize}
\item Let $\alpha$ be an object of $\Hom(\mathbb{Z},G)$, i.e. a groupoid homomorphism from
$\mathbb{Z}$ to $G$. Since for any positive integer $n$,
$\alpha(n)=\alpha(1+\cdots+1)=\alpha(1)\cdot\cdots\cdot \alpha(1)$,
$\alpha(-n)=(\alpha(n))^{-1}$ and $\alpha(0)=1_{s(\alpha(1))}$, $\alpha$ is
uniquely determined by $\alpha(1)\in SG$ (note that $s(\alpha(1))=\alpha(s(1))=\alpha(t(1))=t(\alpha(1))$). Conversely, given any $a\in SG$, the map
$\alpha:n\mapsto a^{n}\ (a^{0}:=1_{s(a)})$ defines a groupoid homomorphism
from $\mathbb{Z}$ to $G$. So there is a bijection between objects of $\Hom(\mathbb{Z},G)$ and $\Lambda G$.
\item Let $\eta:\beta\Leftarrow \alpha$ be an arrow of $\Hom(\mathbb{Z},G)$, i.e. a natural transformation from $\alpha$ to $\beta$. Given any arrow $\bullet \xrightarrow{n}\bullet$ of $\mathbb{Z}$ ($\bullet$ denotes the unique object of $\mathbb{Z}$), by definition, the following diagram commutes
\[ \xymatrix{\alpha(\bullet)\ar[r]^{\alpha(n)}\ar[d]_{\eta_{\bullet}} & \alpha(\bullet)\ar[d]^{\eta_{\bullet}}\\
\beta(\bullet)\ar[r]^{\beta(n)} & \beta(\bullet), } \]
i.e. $\eta_{\bullet}(\alpha(1))^{n}=(\beta(1))^{n}\eta_{\bullet}$, which is equivalent to $\eta_{\bullet}\alpha(1)=\beta(1)\eta_{\bullet}$ and further implies that $(\beta(1),\eta_{\bullet})$ is in $(\Lambda G)_1=SG\rtimes G$. Conversely, given any $(b,g)\in (\Lambda G)_1=SG\rtimes G$, there exists a unique arrow $\eta:\beta\Leftarrow \alpha$
of $\Hom(\mathbb{Z},G)$ such that $\beta(1)=b,\ \alpha(1)=g^{-1}bg$ and $\eta_{\bullet}=g$. Therefore, there
is a bijection between sets of arrows of $\Hom(\mathbb{Z},G)$ and
$\Lambda G$. 
\end{itemize}
It is clear that the bijection above is compatible with all structure maps, leading to an isomorphism of groupoids from $\Hom(\mathbb{Z},G)$ to $\Lambda G$.
\end{proof}
In the sequel, we will identify a functor $\alpha$ from $\mathbb{Z}$ to $G$ with $\alpha(1)\in SG$; furthermore, an arrow $(a,g)$ in $SG\rtimes G$ will be identified with a natural transformation $a\Leftarrow g^{-1}ag$. As a consequence, there is a natural action of $\mathbb{Z}$ (from left) on the space of arrows
$\Hom(\mathbb{Z},G)\cong SG\rtimes G$: the action of $k\in\mathbb{Z}$ on an arrow
$(a,g):a\Leftarrow g^{-1}ag$ is an arrow $(a,a^kg):\ a\Leftarrow(a^{k}g)^{-1}a(a^{k}g)=g^{-1}ag$.
Since $l\cdot(k\cdot(a,g))=l\cdot(a,a^{k}g)=(a,a^{l}a^{k}g)=(l+k)\cdot(a,g)$,
the action is well-defined.

\begin{Rem}
We have ingored the groupoid structure of $\Hom(\mathbb{Z},G)\cong \Lambda G=SG\rtimes G$ in the action defined above. A natural way to extend it to an action of $\mathbb{Z}$ (viewed as a groupoid) on the groupoid $\Lambda G$ is to assume the action on the objects to be identity maps. However, it fails to be an action "by automorphisms", because of the following inequality:
\[ (k\cdot (a,g))(k\cdot (g^{-1}ag,h))\neq k\cdot ((a,g)(g^{-1}ag,h)) .\]
\end{Rem}

Nevertheless, we have the following lemma:
\begin{lem}
The natural map $\rho$ from the product groupoid
$\mathbb{Z}\times \Lambda G\rightrightarrows SG$
to the inertia groupoid $\Lambda G\rightrightarrows SG$
defined by $\rho(k,(a,g))=(a,a^kg)$
is a groupoid homomorphism:
\[ \xymatrix{\mathbb{Z}\times \Lambda G\ar[r]^{\rho}\ar@<-.5ex>[d]\ar@<.5ex>[d] & \Lambda G\ar@<-.5ex>[d]\ar@<.5ex>[d]\\
SG\ar[r]^{\id} & SG
} \]
\end{lem}
In fact, we have the following more general result for crossed modules of groupoids:
\begin{lem}\label{groupoidhom}
Let $N\xrightarrow{\varphi}G$ be a crossed module of groupoids, and let $N\rtimes G\rightrightarrows N$ be the corresponding crossed product groupoid.
Then the natural map $\rho$ from the product groupoid
$\mathbb{Z}\times (N\rtimes G)\rightrightarrows N$
to $N\rtimes G\rightrightarrows N$
defined by $\rho(k,(x,g))=(x,\varphi(x^k)g)$
is a groupoid homomorphism:
\[ \xymatrix{\mathbb{Z}\times (N\rtimes G)\ar[r]^{\rho}\ar@<-.5ex>[d]\ar@<.5ex>[d] & N\rtimes G\ar@<-.5ex>[d]\ar@<.5ex>[d]\\
N\ar[r]^{\id} & N
} \]
\end{lem}
\begin{proof}
It suffices to prove the following:
\begin{equation}\label{homomorphism}
\rho((k,(x,g))(l,(x^g,h)))=\rho((k,(x,g)))\rho((l,(x^g,h))).
\end{equation}
By definition, the left hand side of Equation~\eqref{homomorphism}
is equal to
\[ \rho((k,(x,g))(l,(x^g,h))) = \rho((k+l,(x,gh)))
= (x,\varphi(x^{k+l})gh) = (x,\varphi(x)^{k+l}gh) ,\]
while the right hand side is equal to
\begin{multline*} \rho((k,(x,g)))\rho((l,(x^g,h)))
= (x,\varphi(x^k)g)(x^g,\varphi((x^g)^l)h)
= (x,\varphi(x^k)g\varphi((x^g)^l)h) \\
= (x,\varphi(x^k)g(\varphi(x^g))^lh)
= (x,\varphi(x^k)g(\varphi(x)^g)^lh
= (x,\varphi(x)^kg(g^{-1}\varphi(x)g)^lh \\
= (x,\varphi(x)^kg(g^{-1}\varphi(x)^lg)h
= (x,\varphi(x)^k\varphi(x)^lgh
= (x,\varphi(x)^{k+l}gh)
.\end{multline*}
They are equal to each other. Thus $\rho$ is a groupoid homomorphism.
\end{proof}

\section{A new way to construct the transgression map}

In this section, we will give a new construction of a transgression
map
\[ \Phi:\ H^{\ast}(G_\bullet,\mathbb{Z})\rightarrow H^{\ast-1}((N\rtimes G)_\bullet,\mathbb{Z}) .\]
Finally we will see that it coincides with the transgression map $T_1$ given by Tu-Xu~\cite{TuXu}.

Suppose $N\xrightarrow{\varphi}G$ is a crossed module of groupoids, and let
$N\rtimes G\rightrightarrows N$ be the corresponding crossed product groupoid. We have
a sequence of cohomology group homomorphisms as follows (for simplicity, the coefficient group $\mathbb{Z}$ is omitted):
\[
H^{\ast}(G_\bullet)\xrightarrow{\pi^{*}}H^{\ast}((N\rtimes G)_\bullet)\xrightarrow{\rho^{*}}H^{\ast}((\mathbb{Z}\times(N\rtimes G))_\bullet)\xrightarrow{\cong}\bigoplus_{i+j=\ast}H^{i}(\mathbb{Z}_\bullet)\otimes H^{j}((N\rtimes G)_\bullet)\xrightarrow{\wedge\xi\otimes \id}H^{\ast-1}((N\rtimes G)_\bullet)
\]
where
\begin{itemize}
\item $\pi:\ N\rtimes G\rightarrow G$ is the projection: $\pi(x,g)=g$;
\item $\rho$ is the natural groupoid homomorphism defined in Lemma \ref{groupoidhom};
\item $H^{\ast}((\mathbb{Z}\times(N\rtimes G))_\bullet)\stackrel{\cong}{\longrightarrow}\bigoplus_{i+j=\ast}H^{i}(\mathbb{Z}_\bullet)\otimes H^{j}((N\rtimes G)_\bullet)$
is an application of Eilenberg-Zilber Theorem and K\"unneth formula---see~\cite{Weibel};
\item $\wedge\xi$ is the operation of doing cap product with the generator
$[\xi]$ of $H_{1}(\mathbb{Z}_\bullet)$($=\mathbb{Z}$, as shown in Lemma \ref{Zcohomology}), so $\wedge\xi\otimes \id$ can be thought of as the operation of extracting degree $(\ast -1)$ component from $H^{j}((N\rtimes G)_\bullet)$.
\end{itemize}

\begin{lem}\label{Zcohomology}
  The singular cohomology of $\mathbb{Z}$, viewed as a groupoid, is
  \[ H^i(\mathbb{Z}_\bullet)=\begin{cases}
    \mathbb{Z} & \text{if }i\in\{0,1\}, \\
    0 & \text{otherwise}.
  \end{cases} \]
\end{lem}

The composition of the maps in the sequence is denoted by $\Phi:\ H^{\ast}(G_\bullet)\rightarrow H^{\ast-1}((N\rtimes G)_\bullet)$,
and called the \emph{transgression map} associated to the crossed
module $N\xrightarrow{\varphi}G$.

Our goal is to give an explicit formula for $\Phi$
on the cochain level.

By abbreviating $C^q(G_p;\mathbb{Z})$ as $C^q(G_p)$, $\Phi$ on the level of total complex is given by the composition of the following sequence
\begin{multline*}
\tot_{n}C^{q}(G_{p})\xrightarrow{\pi^{*}}\tot_{n}C^{q}((N\rtimes G)_{p})\xrightarrow{\rho^{*}}\tot_{n}C^{q}((\mathbb{Z}\times(N\rtimes G))_{p})\\
\xrightarrow{\nabla^{*}}\tot_{n}\bigoplus_{\substack{i+j=q \\ k+l=p}}C^{i}(\mathbb{Z}_{k})\otimes C^{j}((N\rtimes G)_{l})\xrightarrow{\wedge\xi\otimes id}\tot_{n}C^{q}((N\rtimes G)_{p-1})
,\end{multline*}
or component-wise, $\Phi$ is indeed the composition of the following horizontal sequence:
\[
C^{q}(G_{p})\xrightarrow{\pi^{*}}C^{q}((N\rtimes G)_{p})\xrightarrow{\rho^{*}}C^{q}((\mathbb{Z}\times(N\rtimes G))_{p})\xrightarrow{\nabla^{*}}\bigoplus_{\substack{i+j=q \\ k+l=p}}C^{i}(\mathbb{Z}_{k})\otimes C^{j}((N\rtimes G)_{l})\xrightarrow{\wedge\xi\otimes id}C^{q}((N\rtimes G)_{p-1}),
\]
where $\nabla^{*}$ is induced by the Eilenberg-MacLane map.
Note that $\wedge\xi\otimes \id$ is only nontrivial for the component
\[ C^q(\mathbb{Z}_{1}\otimes (N\rtimes G)_{p-1})\cong C^{0}(\mathbb{Z}_{1})\otimes C^{q}((N\rtimes G)_{p-1}) ,\]
and the restriction of $\wedge\xi\otimes\id$ to this component is precisely the map
\[ \iota^{*}:\ C^{q}(\mathbb{Z}_{1}\times(N\rtimes G)_{p-1})\rightarrow C^{q}((N\rtimes G)_{p-1}) \]
dual to the inclusion map
\[ \iota:\ (N\rtimes G)_{p-1} \rightarrow\mathbb{Z}\times(N\rtimes G)_{p-1} \]
sending $-$ to $(1,-)$.
So we can reorganize the sequence of cochain complex maps as:
\[
C^{q}(G_{p})\stackrel{\pi^{*}}{\longrightarrow}C^{q}((N\rtimes G)_{p})\stackrel{\rho^{*}}{\longrightarrow}C^{q}((\mathbb{Z}\times(N\rtimes G))_{p})\xrightarrow{\nabla_{1,p-1}^{*}}C^{q}(\mathbb{Z}_{1}\times(N\rtimes G)_{p-1})\xrightarrow{\iota^{*}}C^{q}((N\rtimes G)_{p-1})
\]
where $\nabla_{1,p-1}^\ast$ is the composition of the projection from $\bigoplus_{i+j=q,k+l=p}C^{i}(\mathbb{Z}_{k})\otimes C^{j}((N\rtimes G)_{l})$ to $C^q(\mathbb{Z}_{1}\otimes (N\rtimes G)_{p-1})$ with $\nabla^\ast$, and it is precisely dual to the $(1,p-1)$ component of Eilenberg-MacLane map $\nabla:\ \bigoplus_{k+l=p}C_{q}(\mathbb{Z}_{k}\times(N\rtimes G)_{l})\rightarrow C_{q}((\mathbb{Z}\times(N\rtimes G))_{p})$.

Since $C_q(-)$ means the freely generated module, $\nabla_{1,p-1}$ is indeed equivalent to a map $\phi$ from $\mathbb{Z}_1\times(N\rtimes G)_{p-1}$ to $(\mathbb{Z}\times(N\rtimes G))_{p}$. Applying the formular for Eilenberg-MacLane maps, we obtain
\[
\phi=\sum_{\mu\in\sh(1,p-1)}(-1)^{\mu}\sigma_{\mu(p)}^{h}\cdots\sigma_{\mu(2)}^{h}\sigma_{\mu(1)}^{v}:\ \mathbb{Z}_{1}\times(N\rtimes G)_{p-1}\rightarrow(\mathbb{Z}\times(N\rtimes G))_{p},
\]
where $\sigma_{j}^h=\overline{\eta}^j$ are degeneracy operators of the nerve $\mathbb{Z}_\bullet$ and $\sigma_{i}^v=\overline{\eta}^i$ is a degeneracy operator of the nerve $(N\rtimes G)_\bullet$.

Summarizing the discussion above, $\Phi$ is the dual of the composition of the following sequence of maps:
\[
(N\rtimes G)_{p-1}\xrightarrow{\iota}\mathbb{Z}_{1}\times(N\rtimes G)_{p-1}\xrightarrow{\phi}(\mathbb{Z}\times(N\rtimes G))_{p}\xrightarrow{\rho}(N\rtimes G)_{p}\xrightarrow{\pi}G_{p},
\]
i.e.\ $\Phi=(\pi\circ\rho\circ\phi\circ\iota)^{*}$.

Note that a $(1,p-1)$-shuffle $\mu$ is uniquely determined by $\mu(1)$.
By replacing the subscripts, we obtain:
\[
\phi=\sum_{\mu\in\sh(1,p-1)}(-1)^{\mu}\sigma_{\mu(p)}^{h}\cdots\sigma_{\mu(2)}^{h}\sigma_{\mu(1)}^{v}=\sum_{i=1}^p(-1)^{i}\sigma_{p}^{h}\cdots\hat{\sigma}_{i}^{h}\cdots\sigma_{1}^{h}\sigma_{i}^{v}.
\]

Let $\phi_{i}:=\sigma_{p}^{h}\cdots\hat{\sigma}_{i}^{h}\cdots\sigma_{1}^{h}\sigma_{i}^{v}$,
then for any $(x,g_{1},\cdots,g_{p-1})\in N\rtimes G_{p-1}\cong(N\rtimes G)_{p-1}$,
\begin{multline*}
\phi_{i}(1,x,g_{1},\cdots,g_{p-1})=\sigma_{p}^{h}\cdots\hat{\sigma}_{i}^{h}\cdots\sigma_{1}^{h}\sigma_{i}^{v}(1,x,g_{1},\cdots,g_{p-1})
\\
=(0,\cdots,0,1,0,\cdots,0,x,g_{1},\cdots,g_{i-1},e_{x^{g_{1}\cdots g_{i-1}}},g_{i},\cdots,g_{p-1})
\end{multline*}
where $e_{x^{g_{1}\cdots g_{i-1}}}$ is the identity arrow
in $G$ at $s(\varphi(x^{g_{1}\cdots g_{i-1}}))\in G_0$. For simplicity, we will denote $x^{g_{1}\cdots g_{k}}$  by $x_{k}$.

For an intuitive understanding of the composition of degeneracy operators:
\[
\sigma_{i}^{v}(x,g_{1},\cdots,g_{p-1}):\ \bullet\xleftarrow{g_{1}}\bullet\xleftarrow{g_{2}}\bullet\cdots\bullet\xleftarrow{g_{i-1}}\bullet\overset{e}{\dashleftarrow}\bullet\xleftarrow{g_{i}}\bullet\cdots\bullet\xleftarrow{g_{p-1}}\bullet
\]
is to insert an identity arrow at the $i$-th object;
\[
\sigma_{n}^{h}\cdots\hat{\sigma}_{i}^{h}\cdots\sigma_{1}^{h}(1):\ \bullet\stackrel{0}{\dashleftarrow}\bullet\cdots\bullet\stackrel{0}{\dashleftarrow}\bullet\stackrel{1}{\longleftarrow}\bullet\stackrel{0}{\dashleftarrow}\bullet\cdots\bullet\stackrel{0}{\dashleftarrow}\bullet
\]
is to insert $i-1$ identity arrows on the left and $p-i$ identity
arrows on the right.

With the notations above, the computation can be proceeded as follows
\begin{align*}
   &\quad\rho\circ\phi_{i}\circ\iota(x,g_{1},\cdots,g_{p-1})\\
   &=\rho\circ\phi_{i}(1,x,g_{1},\cdots,g_{p-1})\\
   &=(\rho(0,x,g_{1}),\rho(0,x_1,g_{2}),\cdots,\rho(0,x_{i-2},g_{i-1}),\rho(1,x_{i-1},e_{x_{i-1}}),\rho(0,x_{i-1},g_{i}),\cdots,\rho(0,x_{p-2},g_{p-1}))\\
   &=((x,g_{1}),(x_1,g_{2}),\cdots,(x_{i-2},g_{i-1}),(x_{i-1},\varphi(x_{i-1})\cdot e_{x_{i-1}}),(x_{i-1},g_{i}),\cdots,(x_{p-2},g_{p-1})).
\end{align*}
Let $f_i$ be the composition $\pi\circ\rho\circ\phi_{i}\circ\iota$, then
\begin{align*}
  f_i(x,g_{1},\cdots,g_{p-1})&=\pi\circ\rho\circ\phi_{i}\circ\iota(x,g_{1},\cdots,g_{p-1})\\
  &=(g_{1},\cdots, g_{i-1},\varphi(x_{i-1}),g_{i},\cdots,g_{p-1})\\
   &=(g_{1},\cdots, g_{i-1},\varphi(x)^{g_{1}\cdots g_{i-1}},g_{i},\cdots,g_{p-1})
\end{align*}
As a result, we have
\[
\Phi=(\pi\circ\rho\circ\phi\circ\iota)^{*}=(\sum_{i=1}^{p}(-1)^{i}\pi\circ\rho\circ\phi_{i}\circ\iota)^{*}=\sum_{i=1}^{p}(-1)^{i}f_{i}^{*}.
\]
We see that it is identical to the transgression map (with a sign difference) in Theorem \ref{TuXuResult}, given by Tu-Xu
in~\cite{TuXu}.

\end{document}